\documentclass[11pt,leqno]{amsart}
\usepackage[mathcal]{eucal}
\usepackage[utf8]{inputenc}
\usepackage{amsfonts,amssymb}
\usepackage{hyperref}
\usepackage{graphicx}
\usepackage{enumerate,float}
\usepackage{caption}
\usepackage{xfrac}

\usepackage{tikz}
\tikzset{every picture/.style={line width=0.75pt}} 

\numberwithin{equation}{section}

\theoremstyle{plain}
\newtheorem{theorem}{Theorem}[section]
\newtheorem{lemma}[theorem]{Lemma}
\newtheorem{proposition}[theorem]{Proposition}

\DeclareMathOperator{\Ima}{Im}

\makeatletter
\def\moverlay{\mathpalette\mov@rlay}
\def\mov@rlay#1#2{\leavevmode\vtop{%
   \baselineskip\z@skip \lineskiplimit-\maxdimen
   \ialign{\hfil$\m@th#1##$\hfil\cr#2\crcr}}}
\newcommand{\charfusion}[3][\mathord]{
    #1{\ifx#1\mathop\vphantom{#2}\fi
        \mathpalette\mov@rlay{#2\cr#3}
      }
    \ifx#1\mathop\expandafter\displaylimits\fi}
\makeatother

\makeatletter
\def\author@andify{%
  \nxandlist {\unskip ,\penalty-1 \space\ignorespaces}%
    {\unskip {} \@@and~}%
    {\unskip \penalty-2 \space \@@and~}%
}
\makeatother
\title[Groups with finitely many long commutators of maximal order//Groups with finitely many long commutators of maximal order]{Groups with finitely many long commutators of maximal order}

\author[I. de las Heras]{Iker de las Heras} 
\address{Iker de las Heras: Department of Mathematics, Euskal Herriko Unibertsitatea UPV/EHU, Basque Country}
\email{iker.delasheras@ehu.eus}
\date{}

\author[F. Di Concilio]{Federico Di Concilio} 
\address{Federico Di Concilio: Department of Mathematics, Università degli Studi di Salerno, Italy}
\email{fediconcilio@unisa.it}
\date{}

\author[P. Shumyatsky]{Pavel Shumyatsky} 
\address{Pavel Shumyatsky: Department of Mathematics, University of Brasilia, Brazil}
\email{pavel@unb.br}
\date{}

\makeatletter
\@namedef{subjclassname@2020}{\textup{2020} Mathematics Subject
  Classification}
\makeatother
\thanks{The first author is supported by the Spanish Government, grant PID2020-117281GB-I00 (partly with FEDER funds), and by the Basque Government, grant IT483-22.
During a large part of this research, he was also supported by the European Union via the Marie Skłodowska-Curie Actions (MSCA), grant HE-MSCA-PF-GF22/02 101067088. The second author is partially supported by ``National Group for Algebraic and Geometric Structures, and their Applications'' (GNSAGA-INdAM). The third author acknowledges the support of FAPDF and SNPq.}

\begin{document}

\begin{abstract}
Given a group $G$ and elements $x_1,x_2,\dots, x_\ell\in G$, the commutator of the form $[x_1,x_2,\dots, x_\ell]$ is called a commutator of length $\ell$. The present paper deals with groups having only finitely many commutators of length $\ell$ of maximal order. We establish the following results.\medskip

\noindent Let $G$ be a residually finite group with finitely many commutators of length $\ell$ of maximal order. Then $G$ contains a subgroup $M$ of finite index such that $\gamma_\ell(M)=1$. Moreover, if $G$ is finitely generated, then $\gamma_\ell(G)$ is finite (Theorem \ref{thm: main 1}).\medskip

\noindent Let $\ell,m,n,r$ be positive integers and $G$ an $r$-generator group with at most $m$ commutators of length $\ell$ of maximal order $n$. Suppose that either $n$ is a prime power, or $n=p^{\alpha}q^{\beta}$, where $p$ and $q$ are odd primes, or $G$ is nilpotent.
Then $\gamma_\ell(G)$ is finite of $(m,\ell,r)$-bounded order and there is a subgroup $M\le G$ of $(m,\ell,r)$-bounded index such that $\gamma_\ell(M)=1$ (Theorem \ref{thm: main}).
\end{abstract}

\maketitle

\section{Introduction}

In \cite{CoVe19} Cocke and Venkataraman proved that if $G$ is a group with $m$ elements of maximal order, where $m$ is a positive integer, then the order of $G$ is $m$-bounded. The result gave rise to the work \cite{lmst}, where groups with finitely many commutators of maximal order were studied. The results in \cite{lmst} show that if $G$ is a group with only $m$ commutators of maximal order, then, subject to some additional natural constraints, $G$ has a subgroup $M$ of $m$-bounded index such that the order of the commutator subgroup $M'$ is $m$-bounded.

Throughout this article, we say that a quantity is $(n_1, \dots , n_s)$-bounded to abbreviate “is finite and bounded above in terms of $n_1,\dots,n_s$ only”.

Given a group $G$ and elements $x_1,x_2,x_3,\dots, x_\ell\in G$, we use the simple-commutator notation for left-normed commutators $[x_1,x_2,x_3,\dots, x_\ell]=[...[[x_1,x_2],x_3],\dots ,x_\ell]$. Any commutator of the form $[x_1,x_2,x_3,\dots, x_\ell]$ is called a \emph{commutator of length $\ell$}. 
The present paper deals with groups having at most $m$ commutators of length $\ell$ of maximal order for $\ell\ge 3$, which turned out to be much more challenging than the case $\ell=2$ considered in \cite{lmst}.

Our first result is as follows. As usual, we write $\gamma_i(G)$ to denote the $i$-th term of the lower central series of a group $G$.

\begin{theorem}
\label{thm: main 1}
Let $G$ be a residually finite group with finitely many commutators of length $\ell$ of maximal order. Then $G$ contains a subgroup $M$ of finite index such that $\gamma_\ell(M)=1$. Moreover, if $G$ is finitely generated, then $\gamma_\ell(G)$ is finite.
\end{theorem}

We do not know whether the assumption that $G$ is residually finite is necessary in Theorem \ref{thm: main 1}. Also, it is unclear whether the index of $M$ and the order of $\gamma_\ell(G)$ can be bounded in terms of the relevant parameters. Partial positive answers to these questions are given in the next theorem.

\begin{theorem}
\label{thm: main} Let $\ell,m,n,r$ be positive integers and $G$ an $r$-generator group with at most $m$ commutators of length $\ell$ of maximal order $n$. Suppose that one of the following holds:
    \begin{itemize}
        \item $n$ is a prime-power.
        \item $n=p^{\alpha}q^{\beta}$, where $p$ and $q$ are odd primes and $\alpha,\beta$ are positive integers.
        \item $G$ is nilpotent.
    \end{itemize}
    Then $\gamma_\ell(G)$ is finite of $(m,\ell,r)$-bounded order and there is a subgroup $M\le G$ of $(m,\ell,r)$-bounded index such that $\gamma_\ell(M)=1$.
\end{theorem}

In the next section, we collect some useful lemmas. Proofs of Theorem \ref{thm: main 1} and Theorem \ref{thm: main} are given in Sections 3 and 4, respectively.

\section{Preliminaries}

Throughout the paper, we use standard notation. In particular, $[G,H]$ is the subgroup generated by commutators $[g,h]$, where $g\in G$ and $h\in H$. We write $x^G$ for the conjugacy class containing $x\in G$ and $\langle x\rangle^G$ for the subgroup generated by $x^G$.
Similarly, given a subset $X$ of a group $G$, we write $\langle X\rangle$ for the subgroup generated by $X$, while the symbol $\langle X\rangle^G$ denotes the normal subgroup generated by $X$.

We will freely use the following well-known commutator identities (see for instance \cite[5.1.5]{Ro96}).

\begin{lemma}
    \label{lem: commutators basic}
    Let $x,y,z$ be elements of a group. Then:
    \begin{enumerate}[(i)]
        \item $[x,y]=[y,x]^{-1}$.
        \item $[xy,z]=[x,z]^y[y,z]$ and $[x,yz]=[x,z][x,y]^z$.
        \item $[x,y^{-1}]=[y,x]^{y^{-1}}$ and 
        $[x^{-1},y]=[y,x]^{x^{-1}}$.
    \end{enumerate}
\end{lemma}

As we are concerned with commutators of length $\ell$, the following generalisation of Lemma \ref{lem: commutators basic}(ii) will be extremely useful.

\begin{lemma}
    \label{lem: iker marta 2}
    Let $G$ be a group and let $k\in\{1,\ldots,\ell\}$.
    Take $x_1,\ldots, x_{k-1}, y, x_{k+1}, \ldots, x_\ell \in G$.
    Then there exist $y_{k+1},\ldots,y_{\ell+1}\in \langle y \rangle ^G$ such that
    \begin{align*}
        [x_1,&\ldots, x_{k-1},xy,x_{k+1},\ldots,x_\ell]\\
        &=
        [x_1,\ldots,x_{k-1},x,x_{k+1}^{y_{k+1}},\ldots,x_\ell^{y_\ell}]^{y_{\ell+1}}
        [x_1,\ldots,x_{k-1},y,x_{k+1},\ldots,x_\ell]
    \end{align*}
    for every $x \in G$.
\end{lemma}
\begin{proof} Choose $x\in G$ and 
    observe that
    \begin{align*}
        [x_1,&\ldots, x_{k-1},xy,x_{k+1},\ldots,x_\ell]\\
        &=
        [[x_1,\ldots,x_{k-1},y][x_1,\ldots,x_{k-1},x]^y,x_{k+1},\ldots,x_\ell],\\
        &=
        [[x_1,\ldots,x_{k-1},x]^{y'}[x_1,\ldots,x_{k-1},y],x_{k+1},\ldots,x_\ell],
    \end{align*}
    where $y'=y[x_1,\ldots,x_{k-1},y]^{-1}\in\langle y\rangle^G$.
    Now, by \cite[Lemma 2.8]{HeMo21}, there exist $y_{k+1}$, $\ldots$, $y_{\ell+1}\in\langle[x_1,\ldots,x_{k-1},y]\rangle^G\le\langle y\rangle^G$ such that
    \begin{align*}
        [[x_1,\ldots,x_{k-1}&,x]^{y'}
        [x_1,\ldots,x_{k-1},y],x_{k+1},\ldots,x_\ell]\\
        &=
        [x_1,\ldots,x_{k-1},x,x_{k+1}^{y_{k+1}},\ldots,x_\ell^{y_\ell}]^{y_{\ell+1}}
        [x_1,\ldots,x_{k-1},y,x_{k+1},\ldots,x_\ell].
    \end{align*}
    The lemma follows.
\end{proof}

In what follows we write $Z_i(G)$ to denote the $i$-th term of the upper central series of a group $G$. Recall that an $FC$-group is a group in which every element is contained in a finite conjugacy class. We will require the following classical results, which we collect in a single lemma for the reader's convenience.

\begin{lemma}
    \label{lem: basic results}
    Let $G$ be a group. Then:
    \begin{enumerate}[(i)]
        \item $G$ is a finitely generated $FC$-group if and only if it is a finite extension of its center (see \cite{Neu51}).
    
        \item Every finite normal subset $S$ of $G$ consisting of elements of finite order generates a finite normal subgroup $\langle S\rangle$ of $G$ (Dicman's Lemma).
        Moreover, the order of $\langle S\rangle$ is $(|S|,n)$-bounded, where $n$ is the maximum among the orders of the elements in $S$ (see the last line of the proof of Dicman's Lemma in \mbox{\cite[page~442]{Ro96}}).

        \item For $k\ge 1$, if $G/Z_{k-1}(G)$ is finite, then $\gamma_{k}(G)$ is finite of $|G/Z_{k-1}(G)|$-bounded order (see \cite{Ba52,KuSu13}).
    \end{enumerate}
\end{lemma}

We will also need the following result from  \cite{GuSh15}.

\begin{lemma}
    \label{gush}
    Let $G$ be a finite group and $x\in G$ a commutator of length $\ell$.
    Then every generator of the cyclic subgroup $\langle x\rangle$ is a commutator of length $\ell$.
\end{lemma}

\section{Proof of Theorem \ref{thm: main 1}}

For the rest of the paper, the set of all commutators of length $\ell$ that have maximal order among all commutators of length $\ell$ of a group $G$ will be denoted by $\mathcal{D}_\ell$ (the group $G$ will be clear from the context), while the subgroup generated by $\mathcal{D}_\ell$ will be denoted by $D$.
We will say that an element $g\in G$ is $k$-related to $\mathcal{D}_\ell$, for a non-negative integer $k$, if there exist $g_1,\ldots,g_k\in G$ such that $[g,g_1,\ldots,g_k]\in\mathcal{D}_\ell$.

\begin{proposition}
    \label{prop: infinitely many}
    If a group $G$ contains commutators of length $\ell$ of infinite order, then there are infinitely many such commutators.    
\end{proposition} 
\begin{proof}
    Assume for a contradiction that $G$ has only finitely many commutators of length $\ell$ of infinite order.
    Observe that in that case, the subgroup $D=\langle\mathcal{D}_{\ell}\rangle$ is a finitely generated $FC$-group, and hence the quotient $D/Z(D)$ is finite by Lemma~\ref{lem: basic results}(i). Therefore, $[D,D]$ is finite by Schur's Theorem and so we may assume that $D$ is abelian. Moreover, the torsion part of $D$ is also finite, so we can further assume that $D$ is torsion-free.

    For each $i\ge 0$, define $D_i=[D,G,\overset{i}{\ldots},G]$.
    Observe that every subgroup $D_i$ can be generated by commutators of length $\ell$ of infinite order, so in particular, if $D_i\neq D_{i+1}$ for every $i\ge 0$, there would be infinitely many such commutators.
    Therefore, there must exist $k\ge 0$ such that $D_k=D_i$ for every $i\ge k$.

    If $D_k$ is trivial, set $j=\max\{i\mid D_i\neq 1\}$ so that $D_{j+1}=1$ and $D_j$ is central in $G$.
    It follows that for every $d\in\mathcal{D}_{\ell}$ and $x_1,\ldots,x_j\in G$ we have
    $$
    [d,x_1,\ldots,x_j^n]=[d,x_1,\ldots,x_j]^n.
    $$
    Since $D$ is torsion-free, this contradicts the hypothesis that there are only finitely many commutators of length $\ell$ of infinite order.
    
    Suppose now $D_k$ is non-trivial, and assume without loss of generality that $k\ge \ell-1$.
    Let $d\in D$ and take $x_1,\ldots,x_k\in G$ in such a way that $[d,x_1,\ldots,x_k]\neq 1$.
    Since $D$ is abelian, we have 
    $$
    [d^n,x_1,\ldots,x_k]=[[d,x_1]^n,x_2,\ldots,x_k]=[[d,x_1,x_2]^n,x_3,\ldots,x_k]=\cdots=[d,x_1,\ldots,x_k]^n
    $$
    for every positive integer $n$. Keeping in mind that $k\ge \ell-1$, we deduce that $[d,x_1,\ldots,x_k]^n$ is a non-trivial commutator of length $\ell$.
    Again, since $D$ is torsion-free, we get a contradiction. This completes the proof.
\end{proof}

Hence, according to the previous proposition, if the cardinality of $\mathcal{D}_\ell$ is finite, then the order $n$ of the elements in $\mathcal{D}_\ell$ is finite too.
The following result shows that if $G$ is locally residually finite (for example, if $G$ is finite-by-nilpotent-by-finite or linear) and $|\mathcal{D}_\ell|=m<\infty$, then $n$ is $m$-bounded.

\begin{proposition} 
    \label{prop: bounded order}
    Let $G$ be a locally residually finite group such that $|\mathcal{D}_\ell|=m<\infty$.
    Then the order of the elements in $\mathcal{D}_\ell$ is $m$-bounded.
\end{proposition}
\begin{proof}
    By Proposition \ref{prop: infinitely many} we know that the order of a commutator in $\mathcal{D}_\ell$ is finite, say $n$.
    Let $[x_1,\ldots,x_\ell]\in\mathcal{D}_{\ell}$ and consider the subgroup $H=\langle x_1,\ldots,x_\ell\rangle$ of $G$.
    By hypothesis, $H$ is residually finite, so there exists a normal subgroup $N$ of $H$ of finite index such that $\langle[x_1,\ldots,x_{\ell}]\rangle\cap N={1}$.
    Note that the order of $[x_1,\ldots,x_\ell]N$ in $H/N$ is $n$, and since $H/N$ is finite, by Lemma \ref{gush}, each generator of the cyclic group generated by $[x_1,\ldots,x_\ell]N$ is a commutator of length $\ell$.  It follows that $\phi(n)\leq m$, where $\phi$ denotes Euler's totient function, so, in particular, $n$ is $m$-bounded.
\end{proof}

We now prove the first half of Theorem \ref{thm: main 1}.

\begin{proposition}  \label{prop: residually finite}
    Let $G$ be a residually finite group such that $|\mathcal{D}_\ell|<\infty$.
    Then there exists a subgroup $M\le G$ of finite index such that $\gamma_\ell(M)=1$.
\end{proposition}
\begin{proof}
    By Lemma \ref{lem: basic results}(ii) we know that $D=\langle\mathcal{D}_{\ell}\rangle$ is finite. So since $G$ is residually finite, there exists a normal subgroup $M$ of finite index in $G$ such that $M\cap D=1$.

    We claim that if $k\in\{1,\ldots,\ell\}$ and $x_1,\ldots,x_k\in G$ are elements such that $[x_1,\ldots,x_k]$ is $(\ell-k)$-related to $\mathcal{D}_{\ell}$, then for any $y_1,\ldots,y_k \in M$ the commutator $[y_1x_1,\ldots,y_kx_k]$ is also $(\ell-k)$-related to $\mathcal{D}_{\ell}$.
 
    For a tuple $\mathbf{y}=(y_1,\ldots,y_k)\in M\times\cdots\times M$ write $j_{\mathbf{y}}=\max\{j\mid y_1=\cdots=y_j=1\}$ and we proceed by reverse induction on $j_{\mathbf{y}}$.
    For tuples $\mathbf{y}$ such that $j_{\mathbf{y}}=k$ there is nothing to show, so fix a tuple $\mathbf{y}=(1,\ldots,1,y_j,\ldots,y_k)$ with $j_{\mathbf{y}}=j< k$ and assume that there exist $g_{k+1},\ldots,g_\ell\in G$ such that
    $$
    c:=[x_1,\ldots,x_j,y_{j+1}x_{j+1},\ldots,y_kx_k,g_{k+1},\ldots,g_\ell]\in\mathcal{D}_{\ell}.
    $$
    Write for brevity $z_i=y_ix_i$ for every $i\in\{j,\ldots,k\}$.
    By Lemma \ref{lem: iker marta 2}, there exist $h_{j+1},\ldots,h_{\ell+1}\in G$ such that
    \begin{align*}
        [x_1,\ldots,x_{j-1},z_j,&\ldots,z_k,g_{k+1},\ldots,g_\ell]\\
        &=[x_1,\ldots,x_{j-1},y_j,z_{j+1}^{h_{j+1}},\ldots, z_k^{h_k},g_{k+1}^{h_{k+1}},\ldots,g_\ell^{h_\ell}]^{h_{\ell+1}}c.
    \end{align*}
    This shows that
    $$
    [x_1,\ldots,x_{j-1},z_j,\ldots,z_k,g_{k+1},\ldots,g_\ell]
    \equiv c \pmod{M},
    $$    
    and since $M\cap D=1$, it follows that
    $$
    [x_1,\ldots,x_{j-1},z_{j},\ldots,z_k,g_{k+1},\ldots,g_\ell]\in\mathcal{D}_{\ell},
    $$
    as claimed.

    We will now show that if $k\in\{0,\ldots,\ell\}$ and $x_1,\ldots,x_{k}\in G$ are elements such that $[x_1,\ldots,x_{k}]$ is $(\ell-k)$-related to $\mathcal{D}_{\ell}$, then
    \begin{equation}      
        \label{eq: N nilpotent}
        [x_1,\ldots,x_{k},y_{k+1},\ldots,y_\ell]\in D
    \end{equation}
    for every $y_{k+1},\ldots,y_\ell \in M$. Thus, taking $k=0$, it will follow that $\gamma_\ell(M)\le D$.
    Again, we will proceed by reverse induction on $k$. If $k=\ell$, then the assertion follows trivially, so suppose $k<\ell$.

    Fix $x_1,\ldots,x_{k} \in G$ such that $[x_1,\ldots,x_{k}]$ is $(\ell-k)$-related to $\mathcal{D}_{\ell}$, and assume that for some $x_{k+1}\in G$ the commutator $[x_1,\ldots,x_{k+1}]$ is $(\ell-(k-1))$-related to $\mathcal{D}_{\ell}$.
    Take $y_{k+1},\ldots,y_\ell\in N$.
    We need to show that $d:=[x_1,\ldots,x_{k},y_{k+1},\ldots,y_\ell]\in D$.
    Again, by Lemma \ref{lem: iker marta 2}, there exist $h_{k+1},\ldots,h_{\ell+1}\in N$ such that
    $$
    [x_1,\ldots x_{k},x_{k+1}y_{k+1},y_{k+2},\ldots,y_\ell]=[x_1,\ldots ,x_{k+1},y_{k+2}^{h_{k+2}},\ldots,y_\ell^{h_\ell}]^{h_{\ell+1}}d.
    $$
    We have $x_{k+1}y_{k+1}=y_{k+1}[y_{k+1},x_{k+1}^{-1}]x_{k+1}$ so it follows from the above claim that $[x_1,\ldots ,x_{k},x_{k+1}y_{k+1}]$ is $(\ell-(k-1))$-related to $\mathcal{D}_{\ell}$.
    Hence, inductive hypothesis yields
    $$
    [x_1,\ldots x_{k},x_{k+1}y_{k+1},y_{k+2},\ldots,y_\ell] \in D
    $$
    and
    $$
    [x_1,\ldots,x_{k+1},y_{k+1}^{h_{k+1}},\ldots,y_l^{h_l}]^{h_{\ell+1}}\in D,
    $$
    so we obtain that $d\in D$. This proves (\ref{eq: N nilpotent}). As was already mentioned, considering the case $k=0$ we deduce that $\gamma_\ell(M)\le D$.
    Since $D\cap M=1$, we finally obtain that $\gamma_\ell(M)=1$. The proof is complete.
\end{proof}

The next proposition will be useful in what follows.

\begin{proposition}
    \label{prop: virtually implies finite-by} Let $r$ be a positive integer and $G$ an $r$-generator group such that $|\mathcal{D}_\ell|=m<\infty$.
    Assume that $G$ contains a normal subgroup $M$ of finite index such that $\gamma_\ell(M)$ is finite.
    Then $\gamma_\ell(G)$ is finite of $(m,\ell,r,|G:M|,|\gamma_\ell(M)|)$-bounded order.
\end{proposition}

\begin{proof}
   Since $M$ has finite index in $G$, it is finitely generated with the minimum number of generators bounded in terms of $r$ and $|G:M|$.
    Since $G$ is a finitely generated finite-by-nilpotent-by-finite group, it follows that $G$ is residually finite. Hence, by Proposition \ref{prop: bounded order}, every commutator in $|\mathcal{D}_\ell|$ has $m$-bounded order.
    Let
    $$
    G_\ell=\{[x_1,\ldots,x_\ell]\mid x_1,\ldots,x_\ell\in G\}.
    $$
    Then the subgroup $T$ generated by  $G_\ell\cap M$ is a finite subgroup of $M$ of $(m,r,|G:M|,|\gamma_\ell(M)|)$-bounded order.
    Also, since we have $M/T \leq Z_{\ell-1}(G/T)$, it follows that
    $$
    |G/T:Z_{\ell-1}(G/T)|\le |G:M|.
    $$
    In view of Lemma \ref{lem: basic results}(iii) we conclude that $\gamma_\ell(G/T)$, and hence $\gamma_\ell(G)$, is finite of $(m,\ell,r,|G:M|,|\gamma_{\ell}(M)|)$-bounded order.
\end{proof}

\begin{proof}[Proof of Theorem \ref{thm: main 1}]
    The result is straightforward from Proposition \ref{prop: residually finite} and Proposition \ref{prop: virtually implies finite-by}. 
\end{proof}

\section{Proof of Theorem \ref{thm: main}}

We now introduce some notation to be used in this section.
Let $x_1,\ldots,x_\ell\in G$ and set $d=[x_1,\ldots,x_\ell]$. For every subset $I=\{u_1,\ldots,u_t\}\subseteq\{1,\ldots,\ell\}$ the map $\varphi_I$ is defined as follows: 
\begin{align*}
    \varphi_I : G\times\overset{t}{\cdots}\times G&\longrightarrow G\\
    (y_{u_1},\ldots,y_{u_t})&\longmapsto [z_1,\ldots,z_\ell],
\end{align*}
where
$$
z_i=\left\{
\begin{array}{ll}
    y_i & \text{if }\ i\in I,\\
    x_i & \text{if }\ i\not\in I.
\end{array}
\right.
$$
For brevity, if $(y_{u_1},\ldots,y_{u_t})\in G^t$, we will just write
$$
\varphi_I(y_{u_1},\ldots,y_{u_t})=\varphi_I(y_i),
$$
and if $u\in\{1,\ldots,\ell\}$ and $y_u\in G$, then we will just write $\varphi_{\{u\}}(y_i)=\varphi_u(y_u)$.
We will thus reserve the symbol $i$ for the indices ranging over the subsets (such as $I$ or $\{u\}$) appearing as subscripts of $\varphi$.
Similarly, if we consider subsets $N_{u_1},\ldots,N_{u_t}\subseteq G$, then we will write
$$
\varphi_I(N_i)=\varphi_I(N_{u_1},\ldots,N_{u_t})
$$
and $\varphi_{\{u\}}(N_i)=\varphi_u(N_u)$.

The elements $x_1,\ldots,x_\ell\in G$ used in the definition of $\varphi_I$, for $I\subseteq\{1,\ldots,\ell\}$, will always be clear from the context.

\begin{lemma}
    \label{lem: remove conjugates}
    Let $G$ be a group and let $x_1,\ldots,x_\ell\in G$. Write $d=[x_1,\ldots,x_\ell]$, and suppose there exist $k\in\{0,\ldots,\ell\}$ and $N\trianglelefteq G$ such that $N\le C_G(d)$ and $\varphi_I(N)=1$ for every non-empty subset $I\subseteq \{k+1,\ldots,\ell\}$ (if $k=\ell$, then there is no such $I$).
    Then
    $$
    [x_1,\ldots,x_{k},x_{k+1}y_{k+1},\ldots,x_\ell y_\ell]=d
    $$
    for every $y_{k+1},\ldots,y_\ell\in N$.
\end{lemma}
\begin{proof}
    If $k=\ell$ the lemma follows trivially, so assume $k<\ell$ and suppose, arguing by induction, that 
    $$
    [x_1,\ldots,x_{k+1},x_{k+2}y_{k+2},\ldots,x_\ell y_\ell]=d
    $$    
    for every $y_{k+2},\ldots,y_\ell\in N$.
    Fix $y_{k+1},\ldots,y_\ell\in N$ and observe that by Lemma \ref{lem: iker marta 2}, there exist $h_{k+2},\ldots,h_{\ell+1}\in \langle y_{k+1}\rangle^G\le N$ such that
    $$
    \varphi_{\{k+1,\ldots,\ell\}}(x_iy_i)
    =
    \varphi_{\{k+2,\ldots,\ell\}}\big((x_iy_i)^{h_i}\big)^{h_{\ell+1}}
    [x_1,\ldots,x_{k},y_{k+1},x_{k+2}y_{k+2},\ldots,x_{\ell}y_{\ell}].
    $$
    
    On the one hand, $(x_iy_i)^{h_i}=x_i[x_i,h_i]y_i^{h_i}$ for every $i\in\{k+2,\ldots,\ell\}$, and since $[x_i,h_i]y_i^{h_i}\in N\le C_G(d)$, inductive hypothesis gives
    $$
    \varphi_{\{k+2,\ldots,\ell\}}\big((x_iy_i)^{h_i}\big)^{h_{\ell+1}}=d.
    $$

    On the other hand, write $d^*=[x_1,\ldots,x_{k},y_{k+1},x_{k+2},\ldots,x_{\ell}]$ and observe that since $\varphi_I(N)=1$ for every $I\subseteq\{k+1,\ldots,\ell\}$, we have
    $d^*=1$ and $C_G(d^*)=G$.
    Thus, applying the inductive hypothesis to $d^*$, we obtain
    $$
    [x_1,\ldots,x_{k},y_{k+1},x_{k+2}y_{k+2},\ldots,x_{\ell}y_{\ell}]=d^*=1.
    $$
    It follows that 
    $$
    [x_1,\ldots,x_{k},x_{k+1}y_{k+1},\ldots,x_{\ell}y_{\ell}]=d.
    $$
\end{proof}

Assume the hypotheses of Lemma \ref{lem: remove conjugates}. Thus, for some $k\in\{0,\ldots,\ell\}$ there is $N\trianglelefteq G$ such that $N\le C_G(d)$ and $\varphi_I(N)=1$ for every non-empty $I\subseteq\{k+1,\ldots,\ell\}$.
Then, for any $I\subseteq\{k+1,\ldots,\ell\}$ and any $y_k,y_{k+1},\ldots,y_\ell\in N$, Lemma \ref{lem: iker marta 2} shows that there exist $g_{k+1},\ldots,g_{\ell+1}\in N$ such that
$$
\varphi_{\{k\}\cup I}(x_iy_i)
=
\varphi_{\{k+1,\ldots,\ell\}}(x_ig_i)^{g_{\ell+1}}
\varphi_{\{k\}\cup I}(z_i),
$$
where $z_{k}=y_{k}$ and $z_i=x_iy_i$ if $i>k$, and by Lemma \ref{lem: remove conjugates}, we obtain
\begin{equation}
    \label{eq: stability}
    \varphi_{\{k\}\cup I}(x_iy_i)
    =
    d\varphi_{\{k\}\cup I}(z_i).
\end{equation}
Recall that $D=\langle \mathcal{D}_\ell\rangle$.
We will say that $G$ is \emph{$d$-stable} 
if, whenever such a subgroup $N$ exists for some $k\in\{0,\ldots,\ell\}$, we have
$$
\varphi_{\{k\}\cup I}(x_iy_i)\in D
$$
for every $I\subseteq\{k+1,\ldots,\ell\}$ and every $y_k,\ldots,y_\ell\in N$.
As before, the elements $x_1,\ldots,x_\ell\in G$ will be clear from the context.

As we will see, our proof works not only for the groups satisfying the conditions in Theorem \ref{thm: main}, but for all $d$-stable groups in general.
The following proposition provides some sufficient conditions for a group to be $d$-stable.

\begin{proposition}
    \label{prop: D-stable}
    Let $G$ be a group, and fix $d=[x_1,\ldots,x_\ell]\in\mathcal{D}_\ell$.
    If one of the following conditions holds, then $G$ is $d$-stable:
    \begin{enumerate}[(i)]
        \item The order of $d$ is prime-power.
        \item $G$ is locally nilpotent.
        \item The order of $d$ is $p^{\alpha}q^{\beta}$, where $p$ and $q$ are odd primes and $\alpha,\beta>0$.
    \end{enumerate}
\end{proposition}
\begin{proof} 
    Suppose there exists $k\in\{0,\ldots,\ell\}$ and $N\trianglelefteq G$ with $N\le C_G(d)$ such that $\varphi_I(N)=1$ for every non-empty $I\subseteq\{k+1,\ldots,\ell\}$.
    As in (\ref{eq: stability}), for every $y_k,\ldots,y_\ell\in N$ and every $I\subseteq\{k+1,\ldots,\ell\}$ we have
    \begin{equation}
        \label{eq: stability proof 2}
        \varphi_{\{k\}\cup I}(x_iy_i)=d\varphi_{\{k\}\cup I}(z_i),
    \end{equation}
    where $z_{k}=y_{k}$ and $z_i=x_iy_i$ if $i>k$.
    We need to show that 
    \begin{equation}
        \label{eq: stability proof}
        \varphi_{\{k\}\cup I}(x_iy_i)\in D.
    \end{equation}

    \vspace{5pt}
    
    Suppose first that (i) holds, so that the order of $d$ is $p^\alpha$, where $p$ is a prime and $\alpha$ a non-negative integer.
    If $\varphi_{\{k\}\cup I}(x_iy_i)\in\mathcal{D}_\ell$, then we are done, and if $\varphi_{\{k\}\cup I}(z_i)\in\mathcal{D}_\ell$, then (\ref{eq: stability proof}) follows trivially from (\ref{eq: stability proof 2}).
    Thus, suppose the orders of both $\varphi_{\{k\}\cup I}(x_iy_i)$ and $\varphi_{\{k\}\cup I}(z_i)$ are strictly smaller than $p^{\alpha}$.
    Note, however, that $\varphi_{\{k\}\cup I}(z_i)\in N\le C_G(\mathcal{D}_\ell)$, so that $\varphi_{\{k\}\cup I}(x_iy_i)$ and $\varphi_{\{k\}\cup I}(z_i)$ commute and $d=\varphi_{\{k\}\cup I}(x_iy_i)\varphi_{\{k\}\cup I}(z_i)^{-1}$. Therefore $p^{\alpha}$ must be divisible by the least common multiple of the orders of $\varphi_{\{k\}\cup I}(x_iy_i)$ and $\varphi_{\{k\}\cup I}(z_i)$.
    This is a contradiction, and hence (\ref{eq: stability proof}) is satisfied. 

    \vspace{5pt}

    Suppose now that (ii) holds, and assume first that $G$ is finite.
    Let $P_1,\ldots,P_s$ be the Sylow subgroups of $G$ for some primes $p_1,\ldots,p_s$, and for every $j\in\{1,\ldots,s\}$, let $\mathcal{D}_\ell^{(j)}$ be the set of commutators of length $\ell$ of maximal order in $P_j$.
    We will first show that $\mathcal{D}_\ell^{(j)}\subseteq D$ for every $j\in\{1,\ldots,s\}$.

    Let $n=p_1^{n_1}\cdots p_s^{n_s}$ be the order of $d$, with $n_1,\ldots,n_s$ positive integers, and for every $i \in \{1,\ldots,\ell\}$, write $x_i=\prod_{1\le j\le r}x_{ij}$, where $x_{ij}\in P_j$.
    Then we have
    $$
    d=\prod_{1\le j\le s}[x_{1j},\ldots,x_{\ell j}],
    $$
    and $d_j:=[x_{1j},\ldots,x_{\ell j}]$ has order $p_j^{n_j}$ for every $j\in\{1,\ldots,s\}$.
    Fix now $j^*\in\{1,\ldots,s\}$ and take $d^*=[z_{1},\ldots,z_{\ell}]\in\mathcal{D}_\ell^{(j^*)}$, with $z_{i}\in P_{j^*}$.
    Write also $p_{j^*}^{n^*}$ for the order of $d^*$, and note that, as we have just seen, $n^*\ge n_{j^*}$.
    Set
    $$
    g=d^*\prod_{\substack{1\le j\le s\\ j\neq j^*}}d_j=\Big[z_{1j^*}\prod_{\substack{1\le j\le s\\ j\neq j^*}} x_{1j},\ldots,z_{\ell j^*}\prod_{\substack{1\le j\le s\\ j\neq j^*}} x_{\ell j}\Big],
    $$
    and observe that $g$ is a commutator of length $\ell$ of order
    $$
    p_{j^*}^{n^*}\prod_{\substack{1\le j\le s\\ j\neq j^*}} p_j^{n_j}.
    $$
    This shows that $n^*=n_{j^*}$ and $g\in\mathcal{D}_\ell$.
    Moreover, $d^*$ is a power of $g$, which implies $d^*\in D$.
    As both $j^*$ and $d^*$ were chosen arbitrarily, we conclude that $\mathcal{D}_\ell^{(j)}\subseteq D$ for every $j\in\{1,\ldots,s\}$.

    Now, in order to establish (\ref{eq: stability proof}), write $y_i=\prod_{1\le j\le s}y_{ij}$, and for every $j\in\{1,\ldots,s\}$, define $N_j=N\cap P_j$.
    Then we have $N_j\trianglelefteq P_j$, and since $d_j$ is a power of $d$, we also have $N_j\le C_{P_j}(d_j)$.
    Moreover, if for every $j\in\{1,\ldots,s\}$ we define $\varphi^{(j)}$ in the same way as $\varphi$ but using $d_j=[x_{1j},\ldots,x_{\ell j}]$ instead of $d=[x_1,\ldots,x_\ell]$, then it is clear that $\varphi^{(j)}_I(N_j)=1$ for every $I\subseteq\{k+1,\ldots,\ell\}$.
    Hence, by (i), it follows that $\varphi^{(j)}_{\{k\}\cup I}(x_{ij}y_{ij})\in \langle\mathcal{D}_\ell^{(j)}\rangle\subseteq D$.
    Since
    $$
    \varphi_{\{k\}\cup I}(x_{i}y_{i})
    =
    \prod_{1\le j\le s}\varphi^{(j)}_{\{k\}\cup I}(x_{ij}y_{ij}),
    $$
    the result follows.
    
    Remove now the assumption that $G$ is finite, and let $H=\langle x_1,\ldots,x_\ell,y_1,\ldots,y_\ell\rangle$.
    We observe that $d,\varphi_{\{k\}\cup I}(x_iy_i)\in H$, and therefore, if (\ref{eq: stability proof}) holds in $H$, it will also hold in $G$.
    Thus, we may assume that $G$ is finitely generated, and hence nilpotent.
    In particular, $G$ is residually finite.

    Since $D$ is finite by Proposition \ref{prop: infinitely many} and Lemma \ref{lem: basic results}(ii), there exists $M\trianglelefteq G$ of finite index such that $M \cap D=1$.
    Therefore, the element $dM$ is a commutator of length $\ell$ of maximal order in $G/M$, and $NM/M\trianglelefteq G/M$ with $NM/M\le C_{G/M}(dM)$ such that $\varphi_I(NM/M)=1$ for every non-empty $I\subseteq\{k+1,\ldots,\ell\}$.
    Thus, as we have just shown, $ \varphi_{\{k\}\cup I}(x_{i}y_{i})\in DM$, and since this follows for every $M\trianglelefteq G$ of finite index, we conclude that $ \varphi_{\{k\}\cup I}(x_{i}y_{i})\in D$.
    
    \vspace{5pt}

    Finally, suppose (iii) holds. 
    First, we claim that $d\varphi_{\{k\}\cup I}(z_i)^{-1}$ is a commutator of length $\ell$.
    By Lemma \ref{lem: iker marta 2}, there exist $g_{k+1},\ldots,g_{\ell+1}\in N$ such that
    $$
    1=[x_1,\ldots,x_{k-1},y_k^{-1}y_k,x_{k+1}y_{k+1},\ldots,x_\ell y_\ell]
    =
    \varphi_{\{k,\ldots,\ell\}}(a_i)^{g_{\ell+1}}
    \varphi_{\{k\}\cup I}(z_i),
    $$
    where $a_k=y_k^{-1}$ and $a_i=x_ig_i$ when $i>k$.
    In particular, $\varphi_{\{k\}\cup I}(z_i)^{-1}=\varphi_{\{k,\ldots,\ell\}}(a_i)^{g_{\ell+1}}$.
    Define now $b_k=x_ka_k$ and $b_i=a_i$ for $i>k$, and observe that by Lemma~\ref{lem: iker marta 2} and Lemma~\ref{lem: remove conjugates}, we have
    $$
    \varphi_{\{k,\ldots,\ell\}}(b_i)^{g_{\ell+1}}
    =d\varphi_{\{k,\ldots,\ell\}}(a_i)^{g_{\ell+1}}
    =d\varphi_{\{k\}\cup I}(z_i)^{-1},
    $$
    so the claim follows. 

    Now, write for brevity $g=\varphi_{\{k\}\cup I}(x_iy_i)$, $g'=\varphi_{\{k,\ldots,\ell\}}(b_i)^{g_{\ell+1}}$ and $h=\varphi_{\{k\}\cup I}(z_i)$, so that we have $g=dh$ and $g'=dh^{-1}$.
    As in part (i), we suppose, for a contradiction, that $g,h\not\in\mathcal{D}_\ell$, as otherwise the result follows from (\ref{eq: stability proof}).
    Moreover, since $h\not\in\mathcal{D}_\ell$, it follows from $g'=dh^{-1}$ that $g'\not\in\mathcal{D}_\ell$.
    Also, we can deduce from $h\in N$ that $g, g'$ and $h$ commute, and by a symmetrical argument, we may assume without loss of generality that $p^{\alpha}$ does not divide the order of $h$.
    Now, since $g=dh$, we deduce that $p^{\alpha}q^{\beta}$ divides the least common multiple of the orders of $g$ and $h$, and hence $p^{\alpha}$ must divide the order of $g$.
    For the same reason, $p^{\alpha}$ must also divide the order of $g'$, and since $g,g'\not\in\mathcal{D}_\ell$, it follows that $q^{\beta}$ does not divide the order of $g$ nor the order of $g'$.
    Let $a,b$ be positive integers, both coprime to $q$, such that $g^{q^{\beta-1}a}=(g')^{q^{\beta-1}b}=1$, and set $c=q^{\beta-1}ab$.
    Then,
    $$
    d^{2c}=d^cd^c=(gh^{-1})^c(g'h)^c=g^c(g')^c=1,
    $$
    which is a contradiction because $q$ is odd and $q^{\beta}$ does not divide $2ab$.
    Therefore, $g$ or $h$ must belong to $\mathcal{D}_\ell$, and (\ref{eq: stability proof}) follows again from (\ref{eq: stability proof 2}).
\end{proof}

For the rest of this section, whenever $k\in\{1,\ldots,\ell\}$ we write $J_k$ to denote $\{k,\ldots,\ell\}$ (for convenience, we also write $J_{\ell+1}$ for $\varnothing$).
We write $\mathcal{I}_k$ for $\mathcal{P}(J_{k})$ and define a lexicographic-like order in $\mathcal{I}:=\mathcal{I}_1$ in the following way:
\begin{itemize}
    \item We have $\varnothing<I$ for every non-empty $I\in\mathcal{I}$.
    \item For $I,J\in \mathcal{I}$, if $\min(J)<\min(I)$, then $I<J$.
    \item For $I,J\in \mathcal{I}$, if $\min(J)=\min(I)$, then $I<J$ if and only if $I\setminus\{\min(I)\}<J\setminus\{\min(J)\}$.
\end{itemize} 

The next two lemmas, Lemma \ref{lem: giochetto 1} and Lemma \ref{lem: giochetto 2},  will be key for the inductive step in the proof of Theorem \ref{thm: main(i)}. 

\begin{lemma}
    \label{lem: giochetto 1}
    Let $G$ be a group such that $|\mathcal{D}_\ell|=m<\infty$. Let $d=[x_1,\ldots,x_\ell]\in\mathcal{D}_\ell$ and write $n$ for the order of $d$. Assume that $G$ is $d$-stable. Let $I\subseteq \{1,\ldots,\ell\}$, and suppose there exists a finite index normal subgroup $M\trianglelefteq G$ such that:
    \begin{enumerate}[(i)]
        \item There is a positive integer $b$ such that $|\varphi_J(M)^G|<b$ for every $J<I$.
        \item If $u\in I$, then $\varphi_{(I\setminus J_u)\cup J}(M)=1$ for every non-empty $J\in \mathcal{I}_{u+1}$.
    \end{enumerate}
    Then there exists a normal subgroup $M^*\trianglelefteq G$ of $(m,\ell,b,|G:M|)$-bounded index such that $|\varphi_I(M^*)^G|$ is $(m,n,\ell,b)$-bounded.
\end{lemma}
\begin{proof}
    Let $I=\{u_1,\ldots,u_t\}$, and write $I_j=\{u_1,\ldots,u_j\}$ for every $j\in\{0,\ldots,t\}$ (for $j=0$, we take $I_0=\varnothing$).
    Also, note that (i) implies that the subgroup $C_G(\varphi_J(M)^G)$ has $b$-bounded index in $G$ for every $J<I$, and since $C_G(\mathcal{D}_\ell)$ has $m$-bounded index in $G$, it follows that
    $$
    M^*:=M_G\cap C_G(\mathcal{D}_\ell)\cap\bigcap_{J<I}C_G\big(\varphi_J(M)^G\big)
    $$
    is a normal subgroup of $G$ of $(m,\ell,b,|G:M|)$-bounded index, where $M_G$ stands for the normal core of $M$ in $G$.
    Fix $y_{u_1},\ldots,y_{u_t}\in M^*$.
    We will first prove that for every $s\in\{0,\ldots,t\}$ we have
    \begin{equation}
        \label{eq: varphi complicated}
        \varphi_I(x_iy_i)=\bigg(\prod_{j=0}^{s-1}\varphi_{I_j}(y_i)\bigg) \varphi_I(z_i),
    \end{equation}
    where $z_i=y_i$ for $i \in I_{s}$ and $z_i=x_iy_i$ for $i \in I\setminus I_{s}$.
    
    If $s=0$ then the equality follows trivially, so let $s>0$ and assume by induction that
    \begin{equation}
        \label{eq: varphi induction}
        \varphi_I(x_iy_i)=\bigg(\prod_{j=0}^{s-2}\varphi_{I_j}(y_i)\bigg) \varphi_I(h_i),
    \end{equation}
    where $h_i=y_i$ for $i \in I_{s-1}$ and $h_i=x_iy_i$ for $i \in I\setminus I_{s-1}$.
    Now, applying Lemma \ref{lem: iker marta 2} to $\varphi_I(h_i)$ (taking $k=s$ in the lemma), and removing all the new conjugates by using Lemma \ref{lem: remove conjugates} (taking $d=\varphi_{I_s}(y_i)$, $k=u_s$ and $N=M^*$ in the lemma and using (ii)), we obtain
    $$
    \varphi_I(h_i)=\varphi_{I_{s-1}}(y_i)\varphi_I(z_i),
    $$
    where $h_i$ is as in (\ref{eq: varphi induction}) and $z_i=y_i$ if $i\in I_s$ and $z_i=x_iy_i$ for $i \in I\setminus I_{s}$.
    In particular, (\ref{eq: varphi complicated}) follows for every $s\in\{0,\ldots,t+1\}$, and taking $s=t$ we obtain
    $$
    \varphi_I(x_iy_i)=\prod_{j=0}^{t} \varphi_{I_j}(y_i).
    $$

    Now, since $G$ is $d$-stable, we deduce from (ii) that $\varphi_I(x_iy_i)\in D$, and since $\varphi_I(x_iy_i)=\left(\prod_{j=0}^{t-1} \varphi_{I_j}(y_i)\right)\varphi_I(y_i)$ by (\ref{eq: varphi complicated}), we obtain
    $$
    \varphi_I(y_i)\in\Big(\prod_{j=0}^{t-1}\varphi_{I_j}(M^*)\Big)D.
    $$
    We know from Proposition \ref{prop: bounded order} and Lemma \ref{lem: basic results}(ii) that $|D|$ is $(m,n)$-bounded, and since the $y_i\in M^*$ were  chosen arbitrarily, it follows from (i) that $|\varphi_I(M^*)^G|$ is $(m,n,\ell,b)$-bounded. The proof is complete.
\end{proof}

\begin{lemma}
    \label{lem: giochetto 2}
    Let $G$ be an $r$-generator group and $x_1,\ldots,x_\ell\in G$.
    Set $d=[x_1,\ldots,x_\ell]$.
    Choose $I\subseteq\{1,\ldots,\ell\}$ and write $s=\max(I)$.
    Suppose there exists a finite index subgroup $M\le G$ such that:
    \begin{enumerate}[(i)]
        \item $|\varphi_I(M)^G|<\infty$.
        \item $\varphi_{I\cup J}(M)=1$ for every non-empty $J\in \mathcal{I}_{s+1}$.
    \end{enumerate}
    Then $G$ contains a normal subgroup $M^*$ of finite $(r,\ell,|G:M|,|\varphi_I(M)^G|)$-bounded index such that $\varphi_I(M^*)=1$.
\end{lemma}
\begin{proof}
    Write $I=\{u_1,\ldots,u_t\}$ and let $T=\{u_{t}+1,\ldots,\ell\}$.
    We may assume without loss of generality that $M$ is normal in $G$, and for each $\mathbf{y}=(y_{u_1},\ldots,y_{u_{t-1}}) \in M^{t-1}$ define the map
    \begin{align*}
        \alpha_{\mathbf{y}}:M&\longrightarrow M\\
        y_{u_t}&\longmapsto \varphi_{I}(y_i).
    \end{align*}
    Here $y_{u_t}$ ranges over $M$. 
    
    We will show that $\alpha_{\mathbf{y}}$ is a homomorphism for every $\mathbf{y}\in M^{t-1}$.
    Indeed, by Lemma \ref{lem: iker marta 2}, for every $y_{u_t},g_{u_t} \in M$ there exist $g_{u_t+1},\ldots, g_{\ell}\in M$ such that
    $$
    \varphi_{I}(z_i)=\varphi_{I \cup T}(a_i)\varphi_{I}(b_i),
    $$
    where $z_i=y_i$ if $i \in I\setminus\{u_t\}$ and $z_{u_t}=y_{u_t}g_{u_t}$;
    $a_i=y_i$ if $i \in I$ and $a_i=x_ig_i$ if $i \in T$;
    and $b_i=y_i$ if $i \in I\setminus \{u_t\}$ and $b_{u_t}=g_{u_t}$.
    Now, it follows from (ii) and from Lemma $\ref{lem: remove conjugates}$ that $\varphi_{I \cup T}(a_i)=\varphi_I(y_i)$.
    Hence,
    $$
    \varphi_{I}(z_i)=\varphi_{I}(y_i)\varphi_{I}(b_i),
    $$
    and, as claimed, $\alpha_{\textbf{y}}$ is a homomorphism for every $\textbf{y} \in M^{t-1}$.
    
    For each $\textbf{y}\in M^{t-1}$, let $K_{\mathbf{y}}$ be the kernel of $\alpha_{\textbf{y}}$, and define
    $$
    M^*=\bigcap_{\mathbf{y}\in M^{t-1}} K_{\mathbf{y}}.
    $$
    Since $\Ima(\alpha_{\mathbf{y}})\subseteq \varphi_I(M)$, it follows that $|M:K_{\mathbf{y}}|\le|\varphi_I(M)|$ for every $\mathbf{y}\in M^{t-1}$, and since $G$ is finitely generated, the number of different $K_{\mathbf{y}}$ is $(r,|\varphi_I(M)|)$-bounded (see \cite[Theorem 7.2.9]{mhall}).
    Therefore, $M^*$ has $(r,|\varphi_I(M)|)$-bounded index in $G$ and satisfies \mbox{$\varphi_I(M^*)=1$}.
\end{proof}

We are now ready to complete the proof of Theorem \ref{thm: main}. Indeed, it will directly follow from Proposition \ref{prop: D-stable}, Proposition \ref{prop: virtually implies finite-by} and the following theorem.

\begin{theorem}
    \label{thm: main(i)}
    Let $G$ be an $r$-generator group such that $|\mathcal{D}_\ell|=m<\infty$, and let $d=[x_1,\ldots,x_\ell]\in\mathcal{D}_\ell$. Assume that $G$ is $d$-stable. Then $G$ contains a subgroup $M$ of $(m,\ell,r)$-bounded index such that $\gamma_\ell(M)=1$.
\end{theorem}
\begin{proof} 
    Write $\mathcal{I}=\{I_1,I_2,\ldots,I_{2^\ell}\}$ in such a way that $I_i\le I_j$ if and only if $i\le j$ (that is, $I_1=\varnothing$, $I_2=\{\ell\}$, $I_3=\{\ell-1\}$, $I_4=\{\ell-1,\ell\}$, $I_5=\{\ell-2\}$, and so on). Let $n$ be the order of $d$.
    
    We first claim that there exists a subgroup $M\le G$ of $(m,n,\ell,r)$-bounded index such that $\gamma_\ell(M)=1$.

    In order to prove the claim, we will show that for every $k\in\{1,\ldots,2^\ell\}$ there exists $M_k\trianglelefteq G$ of $(m,n,\ell,r)$-bounded index such that:
    \begin{enumerate}[(i)]
        \item \label{item: subset} $|\varphi_{I_k}(M_k)^G|$ is $(m,n,\ell,r)$-bounded.
        \item \label{item: lemma 1} If $u\in I_k$, then $\varphi_{(I_k\setminus J_u)\cup J}(M_k)=1$ for every non-empty $J\in \mathcal{I}_{u+1}$.
        \item \label{item: lemma 2} If $J_u\subseteq I_k$, then $\varphi_{(I_k\setminus J_u)\cup J}(M_k)=1$ for every non-empty $J\in \mathcal{I}_u$.
    \end{enumerate}
    The claim will thus follow by taking $k=2^\ell$, $u=1$ and $J=J_1$ in (\ref*{item: lemma 2}).

    We will proceed by induction on $k$.
    If $k=1$, then $I_k=\varnothing$ and the assertion follows easily.
    Indeed, (\ref*{item: lemma 1}) and (\ref*{item: lemma 2}) do not apply, and (\ref*{item: subset}) holds since $\varphi_{\varnothing}(M_k)=\{d\}$ and $d^G\subseteq\mathcal{D}_\ell$.
    Suppose then that $k>1$ and that for every $j<k$ there exists $M_j\trianglelefteq G$ of $(m,n,\ell,r)$-bounded index satisfying (\ref*{item: subset}), (\ref*{item: lemma 1}) and (\ref*{item: lemma 2}) with their respective subsets $I_j$.
    Define $M_k=\cap_{j<k} M_j$ and observe that $M_k\trianglelefteq G$ and that $M_k$ has $(m,n,\ell,r)$-bounded index in $G$.  

    \vspace{5pt}

    Let us show that (\ref*{item: lemma 1}) is satisfied.
    Let $u\in I_k$, and let $v$ be minimum such that $J_v\subseteq I_{k-1}$ (note that $v>1$, as otherwise $I_{k-1}=I_{2^\ell}$).
    Then we have
    $$
    I_k=(I_{k-1}\setminus J_v)\cup\{v-1\},
    $$
    and $\max(I_k)=v-1$.
    Now, if $u<v-1$, then $u\in I_{k-1}$, and
    $$
    (I_k\setminus J_u)\cup J=(I_{k-1}\setminus J_u)\cup J
    $$
    for every non-empty $J\in\mathcal{I}_{u+1}$.
    In such a case, the inductive hypothesis on (\ref*{item: lemma 1}) yields $\varphi_{(I_k\setminus J_u)\cup J}(M_k)=1$.
    Assume hence that $u=v-1$.
    Then
    $$
    (I_k\setminus J_u)\cup J=(I_{k-1}\setminus J_v)\cup J
    $$
    for every non-empty $J\in\mathcal{I}_{u+1}=\mathcal{I}_v$, and the inductive hypothesis on (\ref*{item: lemma 2}) now yields $\varphi_{(I_k\setminus J_u)\cup J}(M_k)=1$.
    In any case, (\ref*{item: lemma 1}) follows.

    \vspace{5pt}
    
    Now, it follows by the inductive hypothesis that $|\varphi_J(M_k)^G|$ is $(m,n,\ell,r)$-bounded for every $J<I_k$, so Lemma \ref{lem: giochetto 1} shows there is a subgroup $M^*\le M_k\le G$ of $(m,n,\ell,r)$-bounded index such that (\ref*{item: subset}) holds with $M^*$.
    Abusing notation and redefining $M_k:=M^*$, it is clear that both (\ref*{item: subset}) and (\ref*{item: lemma 1}) hold.

    \vspace{5pt}

    Lastly, we will show that (\ref*{item: lemma 2}) holds.
    Let $u$ be an index such that $J_u\subseteq I_k$ and choose a non-empty subset $J\in \mathcal{I}_u$.
    Suppose first that there exists $v\in \{u,\ldots,\ell\}$ such that $v\not\in J$ and $v<\max(J)$, and define
    $$
    I^*=(I_k\setminus J_u)\cup(J\setminus J_v)\cup \{v\}\ \ \text{ and }\ \ J^*=J\cap J_v.
    $$
    Then we have $I^*<I_k$, $v\in I^*$ and $J^*\in\mathcal{I}_{v+1}$, so by the inductive hypothesis and (\ref*{item: lemma 1}), we obtain that $\varphi_{(I^*\setminus J_v)\cup J^*}(M_k)=1$.
    Since $(I^*\setminus J_v)\cup J^*=(I_k\setminus J_u)\cup J$, this shows that
    \begin{equation}
        \label{eq: varphi=1}
        \varphi_{(I_k\setminus J_u)\cup J}(M_k)=1,
    \end{equation}
    and (\ref*{item: lemma 2}) follows in this case.
    
    Hence, we assume that $J=\{u,u+1,\ldots,v\}$ for some $v\le \ell$.
    If $v=\ell$, then by (\ref*{item: subset}) and Lemma \ref{lem: giochetto 2}, there exists $M^*\trianglelefteq G$ of $(m,n,\ell,r)$-bounded index such that $\varphi_{(I_k\setminus J_u)\cup J}(M^*)=1$, and as before, we may assume $M_k=M^*$.
    Suppose that $v<\ell$ and, by the reverse induction on $v$, that (\ref*{item: lemma 2}) holds whenever $J=\{u,u+1,\ldots,t\}$ with $t>v$.
    By (\ref*{item: subset}) we know that $|\varphi_{(I_k\setminus J_r)\cup J}(M_k)^G|$ is $(m,n,r,\ell)$-bounded, so by Lemma~\ref{lem: giochetto 2} it suffices to show that $\varphi_{(I_k\setminus J_r)\cup J\cup K}(M_k)=1$ for every $K\in\mathcal{I}_{v+1}$ (again, this will produce a new subgroup $M^*\trianglelefteq G$ of $(m,n,\ell,r)$-bounded index for which we may assume that $M_k=M^*)$.
    If $K=\{v+1,v+2,\ldots,t\}$ for some $t\ge v+1$, then the reverse induction on $v$ gives $\varphi_{(I_k\setminus J_u)\cup J\cup K}(M_k)=1$.
    Otherwise, we are in the same situation as in (\ref{eq: varphi=1}), and again we obtain that $\varphi_{(I_k\setminus J_u)\cup J\cup K}(M_k)=1$.
    Hence, (\ref*{item: lemma 2}) holds and the claim follows.
    
    As mentioned at the beginning of the proof, the subgroup $M:=M_{2^\ell}$ is as required.
    
    Observe that since $G$ is a finitely generated virtually nilpotent group, $G$ is residually finite, and therefore $n$ is $m$-bounded by Proposition \ref{prop: bounded order}.
    In particular, $|G:M|$ is $(m,\ell,r)$-bounded.
\end{proof}

\begin{proof}[Proof of Theorem \ref{thm: main}]
    Suppose $G$ satisfies any of the conditions in the statement.
    By Proposition \ref{prop: D-stable}, it follows that $G$ is $d$-stable for some $[x_1,\ldots,x_\ell]\in\mathcal{D}_\ell$, and hence, by Theorem \ref{thm: main(i)}, there exists a subgroup $M\le G$ of $(m,\ell,r)$-bounded index such that $\gamma_\ell(M)=1$.
    This proves the second half of the statement.
    The first half follows directly from Proposition \ref{prop: virtually implies finite-by}.    
\end{proof}

\end{document}